\documentclass[reqno]{amsart}
\usepackage{amssymb,graphics}
\usepackage[mathscr]{euscript}
\usepackage{epsfig}

\newtheorem{thm}{Theorem}[section] 

\newtheorem{lm}[thm]{Lemma}

\newtheorem{prop}[thm]{Proposition}

\newtheorem{claim}{Claim}
\newtheorem{pbm}[thm]{Problem}

\theoremstyle{definition}
\newtheorem{df}[thm]{Definition}

\newcommand{\alg}[1]{{\mathbf #1}}        
\newcommand{\A}{\alg {A}}
\newcommand{\B}{\alg {B}}
\newcommand{\D}{\alg {D}}
\newcommand{\C}{\alg {C}}

\newcommand{\G}{\alg {G}}
\newcommand{\R}{\mathbb{R}}
\renewcommand{\H}{\alg {H}}

\newcommand{\Co}{\alg{Co}(\mathbb{R}^n,A)}
\newcommand{\second}{\alg{Co}(\mathbb{R}^2,A)}

\begin{document}

\author[K.Adaricheva]{Kira Adaricheva$^*$}

\address{Stern College for Women, Yeshiva University, 245 Lexington Ave., New York , NY 10106, USA} \email{adariche@yu.edu}

\begin{abstract}
A closure system with the anti-exchange axiom is called a convex geometry. One geometry is called a sub-geometry of the other if its closed sets form a sublattice in the lattice of closed sets of the other. We prove that convex geometries of relatively convex sets in $n$-dimensional vector space and their finite sub-geometries satisfy the $n$-Carousel Rule, which is the strengthening of the $n$-Carath$\acute{e}$odory property. We also find another property, that is similar to the simplex partition property and does not follow from $2$-Carusel Rule, which holds in sub-geometries of $2$-dimensional geometries of relatively convex sets.  
\end{abstract}

\thanks{\textup{2000} \emph{Mathematics Subject Classification}: 05B25,06B15, 06B05, 51D20}

\thanks{$^*$Stern College for Women, Yeshiva University, 245 Lexington Ave., New York, NY 10016, USA}

\title[Representing finite convex geometries]{Representing finite convex geometries by relatively convex sets}
\maketitle

\section{Introduction}
A closure system $\A=(A, -)$, i.e. a set $A$ with a closure operator $-:2^A\rightarrow 2^A$ defined on $A$, is called \emph{a convex geometry} (see \cite{AGT}), if it is a zero-closed space (i.e. $\overline{\emptyset}=\emptyset$) and it satisfies \emph{the anti-exchange axiom}, i.e.
\[
\begin{aligned}
x\in\overline{X\cup\{y\}}\text{ and }x\notin X
\text{ imply that }y\notin\overline{X\cup\{x\}}\\
\text{ for all }x\neq y\text{ in }A\text{ and all closed }X\subseteq A.
\end{aligned}
\]

A convex geometry $\A=(A,-)$ is called finite, if set $A$ is finite.

Very often, a convex geometry can be represented by its collection of closed sets. There is a convenient description of those collections of subsets of a given finite set $A$, which are, in fact, the closed sets of a convex geometry on $A$: if $\mathcal{A} \subseteq 2^A$ satisfies\\
(1) $\emptyset \in \mathcal{A}$;\\
(2) $X\cap Y \in \mathcal{A}$, as soon as $X,Y \in \mathcal{A}$;\\
(3) $X \in \mathcal{A}$ and $X\not = A$ implies $X \cup \{a\} \in \mathcal{A}$, for some $a \in A\setminus X$,\\
then $\mathcal{A}$ represents the collection of closed sets of a convex geometry $\A=(A,\mathcal{A})$.

A reader can be referred to \cite{D},\cite{EdJa} for the further details of combinatorial and lattice-theoretical aspects of finite convex geometries.
 
For convex geometries $\A=(A,-)$ and $\B=(B,\tau)$, one says that $\A$ is a sub-geometry of $\B$, if there is a one-to-one map $\phi$ of closed sets of $\A$ to closed sets of $\B$ such that $\phi(X\cap Y)=\phi(X)\cap \phi(Y)$, and $\phi(\overline{X\cup Y})=\tau(\phi(X)\cup \phi(Y))$, where $X,Y \subseteq A$, $\overline{X}=X$, $\overline{Y}=Y$. In other words, the lattice of closed subsets of $\A$ is a sublattice of the lattice of closed sets of $\B$. When geometries $\A$ and $\B$ are defined on the same set $X=A=B$, we also call $\B$ \emph{a strong extension} of $\A$. Extensions of finite convex geometries were considered in \cite{AdNa} and \cite{AGT}, the more systematic treatment of extensions of finite lattices was given in \cite{Na}.

Given any class $\mathcal{L}$ of convex geometries, we will call it \emph{universal}, if an arbitrary finite convex geometry is a sub-geometry of some geometry in $\mathcal{L}$.

One of main results in \cite{AGT} proves that a specially designed class of convex geometries $\mathcal{AL}$ is universal. Namely, $\mathcal{AL}$ consists of convex geometries of the form $Sp(A)$, each of which is built on a carrier set of an algebraic and dually algebraic lattice $A$ and whose closed sets are all complete lower subsemilattices of $A$ closed with respect to taking joins of non-empty chains. At the same time, a subclass of all \emph{finite} geometries from class $\mathcal{AL}$ cease to be universal,
see \cite{Ad2} and \cite{AGT}.

In this paper, we want to consider another conveniently designed class of convex geometries, in fact, even an infinite hierarchy of classes.

Given a set of points $A$ in Euclidean $n$-dimensional space $\mathbb{R}^n$, one defines a closure operator $-:2^A\rightarrow 2^A$ on $A$ as follows: for any $Y \subseteq A$, $\overline{Y}= ch(Y) \cap A$, where $ch$ stands for \emph{the convex hull}. One easily verifies that such an operator satisfies the anti-exchange axiom. Thus, $(A,-)$ is a convex geometry, which also will be denoted as $\Co$. We will call such convex geometry \emph{a geometry of relatively convex sets} (assuming that these are convex sets ``relative'' to $A$). 
The convex geometries of relatively convex sets were studied in \cite{Hu},\cite{Be} and \cite{Ad}.

For any geometry $C=\alg{Co}(\mathbb{R}^m,A)$, we will call $n \in \mathbb{N}$ \emph{a dimension} of $C$, if $n$ is the smallest number such that $C$ could be represented as $\alg{Co}(\mathbb{R}^n,A)$, for appropriate $A \subseteq R^n$. In particular, $n \leq m$, and $n \leq p-1$, if $A$ is a finite non-empty set of cardinality $p>1$.

Let $\mathcal{C}_n$ be the class of convex geometries of relatively convex sets of dimension $\leq n$, and let $\mathcal{C}$ be the the class of of all convex geometries of relatively convex sets of finite dimension (thus, including $\mathcal{C}_n$, $n \in \mathbb{N}$, as subclasses). By $\mathcal{C}_B$ we denote a subclass of $\mathcal{C}$ that consists of geometries of convex sets relative to bounded sets, i.e. $\alg{Co}(\mathbb{R}^n,A)$, for some $n$ and $A\subseteq B$, where $B$ is a ball in $R^n$. By $\mathcal{C}_f$ we denote a subclass if \emph{finite} convex geometries in $\mathcal{C}$.

It is known that none of $\mathcal{C}_n$ is universal, due to the $n$-Carath$\acute{\text{e}}$odory property that holds on any sub-geometry of geometry from $\mathcal{C}_n$ (see, for example, \cite{Be}), but fails on any geometry of dimension $n+1$. We introduce a stronger property called the $n$-Carousel Rule and show that it holds on sub-geometries of $\mathcal{C}_n$. It allows to build a series of finite convex geometries $C_n$ such that $C_n$ satisfies the $n$-Carath$\acute{\text{e}}$odory property, but cannot be a sub-geometry of any geometry in $\mathcal{C}_n$. On the other hand, $C_n$ is a sub-geometry of some geometry in $\mathcal{C}_{n+1}$. We also prove that the
so-called Sharp Carousel Rule holds in all sub-geometries in $\mathcal{C}_2$, a slight modification of \emph{the simplex partition property} from \cite{MoSo}.

It was shown in \cite{Be} that \emph{every finite} closure system can be embedded into some geometry in the class $\mathcal{C}$, in particular, this class is universal for all finite convex geometries. This observation is a direct consequence of deep and complex result proved in \cite{PuTu} that every finite lattice is a sublattice of a finite partition lattice. Thus, class $\mathcal{C}$ can not be considered as specific to finite convex geometries. It is worth noting that the construction in \cite{Be} uses convex sets relative to $A$ which is the collection of lines, in particular, $A$ is always an unbounded set. 

This leaves the following open questions:

\begin{pbm}\label{first}
Is class $\mathcal{C}_B$ of geometries of convex sets relative to bounded sets universal? Is the class
$\mathcal{C}_f$ of finite geometries of relatively convex sets universal?
\end{pbm}

Note that the second question of two is a modification of Problem 3 from \cite{AGT}.

\section{Carath$\acute{\text{\smaller{E}}}$odory Property and Carousel Rule}

We recall that a convex geometry $(A,-)$ satisfies the $n$-\emph{Carath}$\acute{e}$\emph{odory property}, if
$x \in \overline{S}$, $S \subseteq A$, implies $x \in \overline{\{a_0,\dots,a_n\}}$ for some $a_0,\dots,a_n \in S$. Besides, $a_0$ can be taken to be any pre-specified element of $S$.

\begin{prop}\emph{(\cite[Lemma 3.2]{Hu},\cite[Proposition 25]{Be})} For any $n \in \mathbb{N}$ and $A \subseteq \R^n$, convex geometry $\Co$ satisfies the $n$-Carath$\acute{\text{e}}$odory property.
\end{prop}

Our aim is to formulate a stronger property, which we call the $n$-\emph{Carousel Rule}, extending to arbitrary finite dimensions the $2$-Carousel Rule introduced in \cite{AW1}.

\begin{df} A convex geometry $(A,-)$ satisfies the $n$-Carousel Rule, if $x,y \in \overline{S}$,
$S \subseteq A$, implies $x \in \overline{\{y, a_1,\dots,a_n\}}$ for some $a_1,\dots,a_n \in S$.
\end{df}

Note that the $n$-Carath$\acute{\text{e}}$odory property follows from the $n$-Carousel Rule. Indeed, if $y$ is chosen among elements of $S$, and $x \in \overline {S}$, then, according to the $n$-Carousel Rule, $x \in \overline{\{y, a_1,\dots,a_n\}}$ for some $a_1,\dots,a_n \in S$, which is also a desired conclusion for the $n$-Carath$\acute{\text{e}}$odory property.

\begin{lm}\label{Co with Rule} For any $n \in \mathbb{N}$ and $A \subseteq R^n$, convex geometry $\Co$ satisfies the $n$-Carousel Rule.
\end{lm}

\begin{proof} Consider $\G=\Co$, and let $x,y \in \overline {S}$, for some $S \subseteq A$. 

Due to the $n$-Carath$\acute{\text{e}}$odory property, $x \in \overline{\{c_0, c_1,\dots,c_n\}}$ and $y \in \overline{\{b_0, b_1,\dots,b_n\}}$ for some
$c_0,b_0,\dots, c_n,b_n \in S$. In other words, points $x,y$ belong to a convex polytope $P$ in $R^n$ with the vertices among $c_0,b_0,\dots, c_n,b_n$. Suppose $F_1,\dots, F_k$ are the faces of this polytope, i.e. they are at most $(n-1)$-dimensional convex polytopes. 
For arbitrary $y \in P$, we have $P \subseteq \bigcup_{i\leq k} P_i$, where $P_i=ch(y\cup F_i)$, $i=1,\dots,k$. Hence, $x \in \overline {y \cup F_i}$ for some $i\leq k$. Now, due to the $n$-Carath$\acute{\text{e}}$odory property, $x \in \overline {\{y, f_1,\dots,f_n\}}$ for some vertices $f_1,\dots,f_n$ of $F_i$, which are also elements of $S$.
Thus, the conclusion of the $n$-Carousel Rule holds. 
\end{proof}

Our next goal is to show that the $n$-Carousel Rule is preserved on finite sub-geometries.

\begin{lm}\label{Sub} If geometry $\H$ satisfies the $n$-Carousel Rule, and $\G$ is a finite sub-geometry of $\H$, then $\G$ satisfies the $n$-Carousel Rule.
\end{lm}
\begin{proof} Suppose $\H=(H,-)$, $\G=(G,\tau)$ and $\phi$ is a one-to-one mapping from closed sets of $\G$ to closed sets of $\H$ that preserves the intersection and the closure of finite unions of sets.

Let assume that $\G$ does not satisfy the $n$-Carousel Rule. It means that, for some $x,y \in G$ and $S \subseteq G$, we have $x,y \in \overline{S}$, but $x \not \in \overline{\{y, s_1,\dots,s_n\}}$, for any $s_1,\dots,s_n \in S$. In any finite convex geometry $(A,-)$, for any $a \in A$, the subset $\overline{a}\setminus a$ is closed. Hence, set $X=\overline{x}\setminus x$ is closed in $\G$. According to our assumption, $\overline{x}\cap \overline{\overline{y}\cup \overline{s_1}\cup \dots\cup\overline{s_n}} \subseteq X$, for any $s_1,\dots, s_n \in S$.

Take $x' \in \phi(\overline{x})\setminus \phi(X)$ and $y'\in \phi(\overline{y})$. 
Note that $\overline{S}=\bigcup\{\overline{s}: s \in S\}$, hence $S'=\phi(\overline{S})=\tau(\bigcup(\phi(\overline{s}): s \in S))$. Since $x',y' \in S'=\tau(\bigcup(\phi(\overline{s}): s \in S))$ and $\H$ satisfies the $n$-Carousel Rule, we have $x' \in \tau (\{y',s'_1,\dots,s'_n\})$, for some $s'_i\in \phi(\overline{s_i})$, $s_i \in S$. It follows $x' \in \phi(\overline{x})\cap \tau(\phi(\overline{y})\cup \phi(\overline{s_1})\cup \dots \cup \phi(\overline{s_n}))=\phi(\overline{x})\cap \phi(\overline{\overline{y}\cup \overline{s_1}\cup \dots\cup\overline{s_n}})$, which means $\phi(\overline{x}\cap \overline{\overline{y}\cup \overline{s_1}\cup \dots\cup\overline{s_n}}) \not\subseteq \phi(X)$, a contradiction.
\end{proof}

\section{Convex geometries $C_n$}

Using the $n$-Carousel Rule, it will not be difficult to built an example of a finite convex geometry that cannot be a sub-geometry of relatively convex sets of dimension $\leq n$.

Consider a point configuration in $\R^n$ that consists of extreme points $a_0,\dots,a_n$, equivalently, the vertices of a $n$-dimensional polytope $P$, and inner points $x,y$ of $P$. Besides, choose $x,y$ so that $x$ belongs to only one of polytopes $P_i= ch(\{y\}\cup D\setminus a_i)$ and $y$ belongs to only one of polytopes $Q_j=ch(\{x\}\cup D\setminus a_j)$, where $D=\{a_0,\dots,a_n\}$, $i,j\leq n $.

Let $\D_n= \alg{Co}(\mathbb{R}^n,D\cup\{x,y\})$.

According to our assumption, $\{y\}\cup D\setminus a_i$ and $\{x\}\cup D\setminus a_j$ are not closed sets in convex geometry $D_n$, for some unique $i,j\leq n,i\not = j$.

Consider closure space $\C_n=(D\cup\{x,y\},\mathcal{D})$, where a family of closed sets
$\mathcal{D}$ is defined as a collection of all closed sets of convex geometry $\D_n$, plus sets
$\{y\}\cup D\setminus a_i$ and $\{x\}\cup D\setminus a_j$. These are, indeed, the closed sets of a closure operator, since the intersection of any members of $\mathcal{D}$ is again in $\mathcal{D}$. For this, it is enough to note that any subset of $\{y\}\cup D\setminus a_i$ and $\{x\}\cup D\setminus a_j$ is a closed set of convex geometry $\D_n$. We can claim more, namely:

\begin{lm}
$\C_n$ is a (finite) convex geometry that satisfies the $n$-Carath$\acute{\text{e}}$odory property.
\end{lm}
\begin{proof}
To show that $\C_n$ is a convex geometry, one needs to demonstrate that every closed set can be extended by one point to obtain another closed set. This is true for any closed set of $\D_n$, since it is a convex geometry itself. This is also true for additional sets $\{y\}\cup D\setminus a_i$ and $\{x\}\cup D\setminus a_j$: the first can be extended by $x$ to obtain $\{x, y\} \cup D\setminus a_i$, a closed set of $\D_n$, the second can be extended by $y$ to obtain $\{x,y\} \cup D\setminus a_j$, another closed set of $\D_n$.
\end{proof}

\begin{lm}
$\C_n$ cannot be a sub-geometry of any geometry of relatively convex sets of dimension $\leq n$.
\end{lm}
\begin{proof} Indeed, $\C_n$ does not satisfy the $n$-Carousel Rule, since $x$ is not in a closure of $y$ with any $n$ points from $D$ (similarly, $y$ is not in a closure of $x$ with any $n$ points from $D$). Hence, the claim of this lemma follows from \ref{Co with Rule} and \ref{Sub}. 
\end{proof}

On the other hand, we can show that $\C_n$ is a sub-geometry of some $(n+1)$-dimensional geometry of relatively convex sets.
Indeed, consider $\R^{n+1}$, and subspace $S_0 \subseteq \R^{n+1}$ of all points whose last projection is $0$; correspondingly, let $S_1 \subseteq \R^{n+1}$ be a subspace of all points whose last projection is $1$. Consider points $c_0,c_1,\dots, c_n \in S_0$ whose convex hull is $n$-dimensional polytope $C$, and take an inner point $u$ of $C$. Let $b_0,b_1,\dots,b_n,v \in S_1 $ be obtained from $c_0,c_1,\dots, c_n, u$, correspondingly, by replacing the last projection by $1$. Let $K=\{c_0,b_0,c_1,b_1,\dots,c_n,b_n,u,v\}$ and $\G_{n+1}=\alg{Co}(\mathbb{R}^{n+1},K)$.

Define a mapping $\phi$ from closed sets of $\C_n$ to closed sets of $\G_{n+1}$:
$\phi(\{a_i\})=\{c_i,b_i\}$, $i=0,\dots,n$, $\phi(\{x\})=\{u\}$, $\phi(\{y\})=\{v\}$.
For any closed set $S=\{s_1,\dots,s_k\}$, $k>1$, of $\C_n$, it is straightforward to check that $\phi(s_1)\cup\dots\cup \phi(s_k)$ is closed in $\G_{n+1}$, thus, we may define $\phi(S)= \phi(s_1)\cup\dots\cup \phi(s_k)$, for any closed $S$ in $\C_n$. Evidently, this mapping preserves intersections. As for the closure of a union of closed sets $X,Y$ in $\C_n$, we observe that

$\overline{X\cup Y} = \left\{
\begin{array}{c l}
  X\cup Y \cup \{x,y\}, & a_0,\dots,a_n\in X\cup Y; \\
  \\
  X\cup Y, & \mbox{ otherwise}. \\
\end{array}
\right.
$

Similarly, in $\G_{n+1}$, $u$ (symmetrically, $v$) is not in a closure of $v$ ($u$) with any $n$ sets $\{c_i,b_i\}$, $i=0,\dots,n$, since $u$ ($v$) is an inner point of $n$-dimensional polytope with vertices $c_0,\dots,c_n$ ($b_0,\dots,b_n$). Hence, we have in $\G_{n+1}$\\
 
$\overline{\phi(X)\cup \phi(Y)} = \left\{
\begin{array}{c l}
  \phi(X)\cup \phi(Y) \cup \{u,v\}, & a_0,\dots,a_n\in X\cup Y; \\
  \\
  \phi(X)\cup \phi(Y), & \mbox{ otherwise}. \\
\end{array}
\right.
$

Therefore, $\phi$ preserves the closure of the union of closed sets, too.

\section{Sharpening Carusel Rule}

It turns our that one can slightly strengthen the $n$-Carusel Rule, and we are going to illustrate it in case of $2$-Carusel Rule.

\begin{figure}[ht]
\begin{center}
\includegraphics[scale=.45]{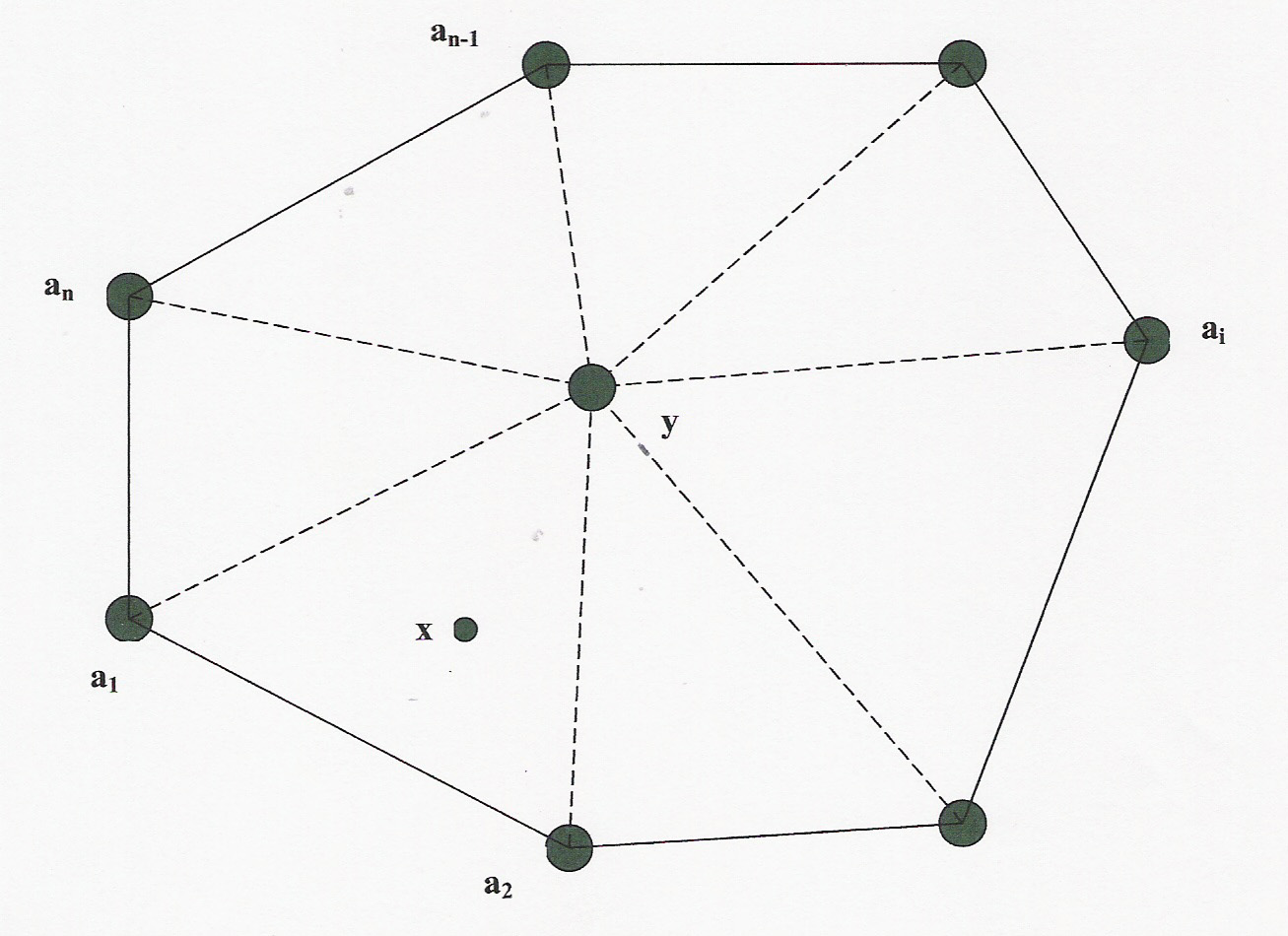}
\caption{}
\end{center}
\end{figure}

First of all, let see the visual image of $2$-Carusel Rule on Figure 1: if $x,y$ are in the convex polygone generated by $a_1,\dots,a_n$, then $x$ should be at least in one triangle generated by $y$ and two points from $a_1,\dots,a_n$. In general, there might be multiple triangles of that sort containing $x$. On the other hand, if $n=3$, i.e. $x,y$ are inside the triangle defined by $a_1,a_2,a_3$, $x$ can belong to maximum two triangles. In this case, $x$ will be also on the segment containing $y$ and one of points $a_1,a_2,a_3$. Indeed, if, say, $x \in \overline{\{y,a_1,a_2\}}$ and $x \in \overline{\{y,a_1,a_3\}}$, then $x \in \overline{\{y,a_1\}}$. Note that the property will hold even if $y$ belongs to the boundary of triangle $a_1,a_2,a_3$.

The version of this property under additional assumption that the points on the plane are in the \emph{general position}, i.e. no three of them are on the same line, is called \emph{the simplex partition property} in \cite{MoSo}. In this case, one would say that $x$ can be in exactly one of triangles $\overline{\{y,a_i,a_j\}}$,
$i,j \in \{1,2,3\}$.

 It turns out we can make the similar statement in any sub-geometry of $2$-dimensional geometry, as long as we assume that $y$ is not on the boundary of $a_1,a_2,a_3$. 

\begin{thm}\label{sharp}
Let $\G=(G,-)$ be any sub-geometry of $2$-dimensional finite geometry $\G_0=\alg{Co}(\mathbb{R}^2,G_0)$. Then the following implication holds for all $x,y,a,b,c$ in $G$: if $y \in \overline{\{a,b,c\}}$, $\overline{y}\cap \overline{\{a,b\}}= \overline{y}\cap\overline{\{b,c\}}= \overline{y}\cap\overline{\{a,c\}}=\overline{y}\cap \overline{x}=\emptyset$, $x\in \overline{\{y,a,b\}}$ and $x \in \overline{\{y,a,c\}}$, then $\overline{x} \cap \overline{\{y,a\}}>\emptyset$. 
\end{thm}

To prove Theorem we will need a few auxiliary statements.

\begin{lm}\label{s in 3} Let $a_1,\dots,a_i,\dots,a_j,\dots,a_k,\dots,a_s,\dots,a_n$ be a circular order of vertices of some convex polygon on the plane. If $s$ is a point of intersection of segments $[a_1,a_j]$ and $[a_i,a_k]$, then $s$ is in triangle $\overline{\{a_s,a_i,a_j\}}$.
\end{lm}

\begin{figure}[ht]
\begin{center}
\includegraphics[scale=.45]{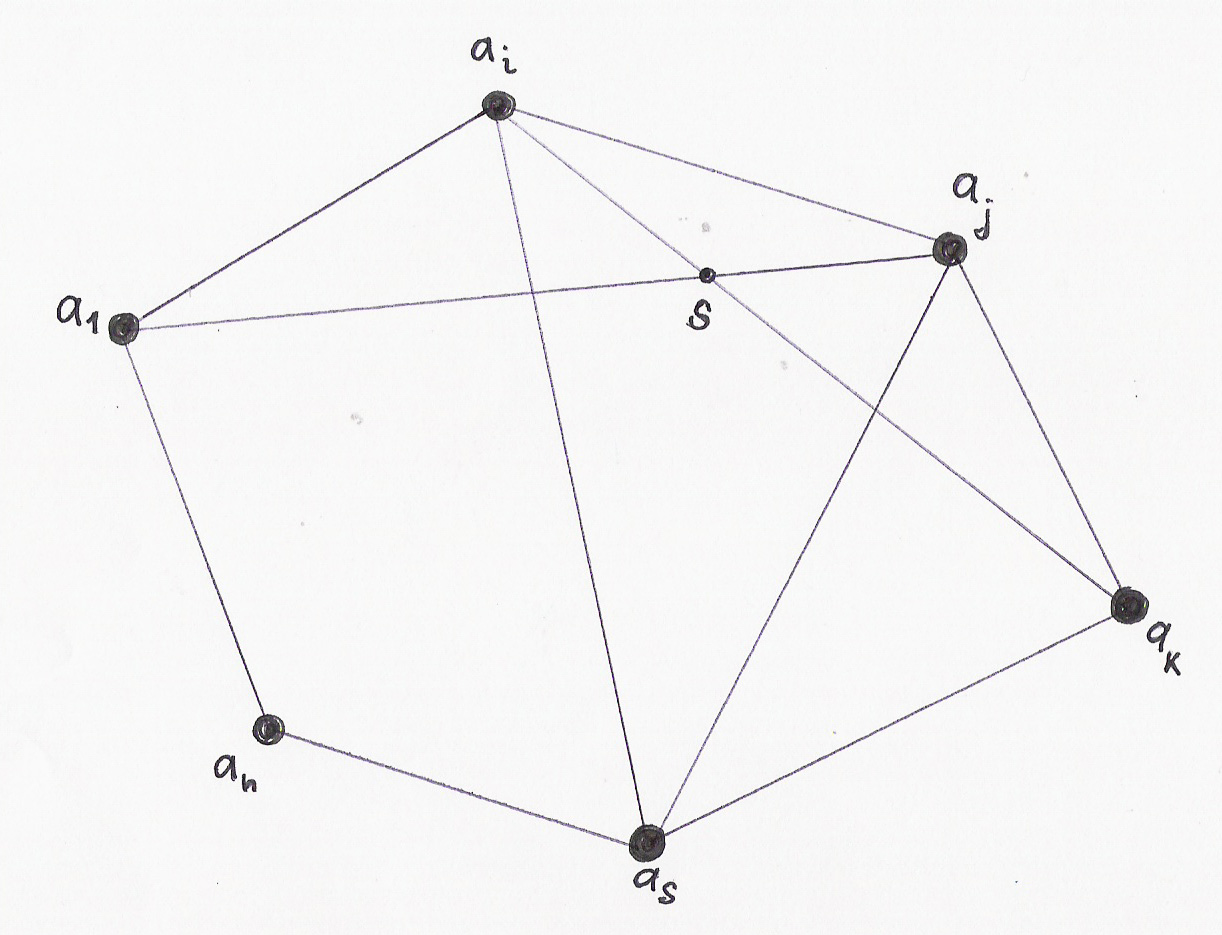}
\caption{}
\label{s}
\end{center}
\end{figure}

\begin{proof}
It is true for any ''diagonal'' of a convex polygon $[a_1,a_j]$ that all the vertices between $a_j$ and $a_1$ in their circular order belong to the same semi-plane generated by the line $(a_1,a_j)$. In particular, $[a_1,a_j]$ and $[a_i,a_k]$, indeed, intersect at some point $s$, since the points $a_i$ and $a_k$ are separated by line $(a_1,a_j)$.

In order to show that $s$ is inside triangle $\overline{\{a_s,a_i,a_j\}}$, one needs to show that, for each side of a triangle, the third vertex and point $s$ belong to the same semiplane generated by the line extending this side.

Take side $[a_i,a_j]$, then vertices $a_k,a_s,a_1$ are in the same semiplane generated by line $(a_i,a_j)$, hence, both segments $[a_1,a_k]$ and $[a_i,a_j]$ are in that semiplane, implying that their intersection point $s$ belongs there as well.

Take another side of triangle $[a_i,a_s]$. Then $a_j,a_k$ are in the same semiplane generated by line $(a_i,a_j)$. Since $s$ is on segment $[a_i,a_k]$, it belongs to the same semiplane. Thus, $s$ and $a_j$ belong to the same semiplane generated by $(a_i,a_s)$, which is needed. Similar is true for the side $[a_j,a_s]$ and points $a_i$ and $s$. 
\end{proof}

\begin{lm}\label{abc}
Suppose the vertices of a convex polygon $M$ with at least 4 vertices are split into three subsets $A,B,C$. If the vertices of one of these subsets are separated by the vertices of the others in the circular order, then every point of convex polygon $M$ belongs to $\overline{A\cup B}\cup\overline{A\cup C}\cup\overline{B\cup C}$.
\end{lm}
\begin{figure}[ht]
\begin{center}
\includegraphics[scale=.45]{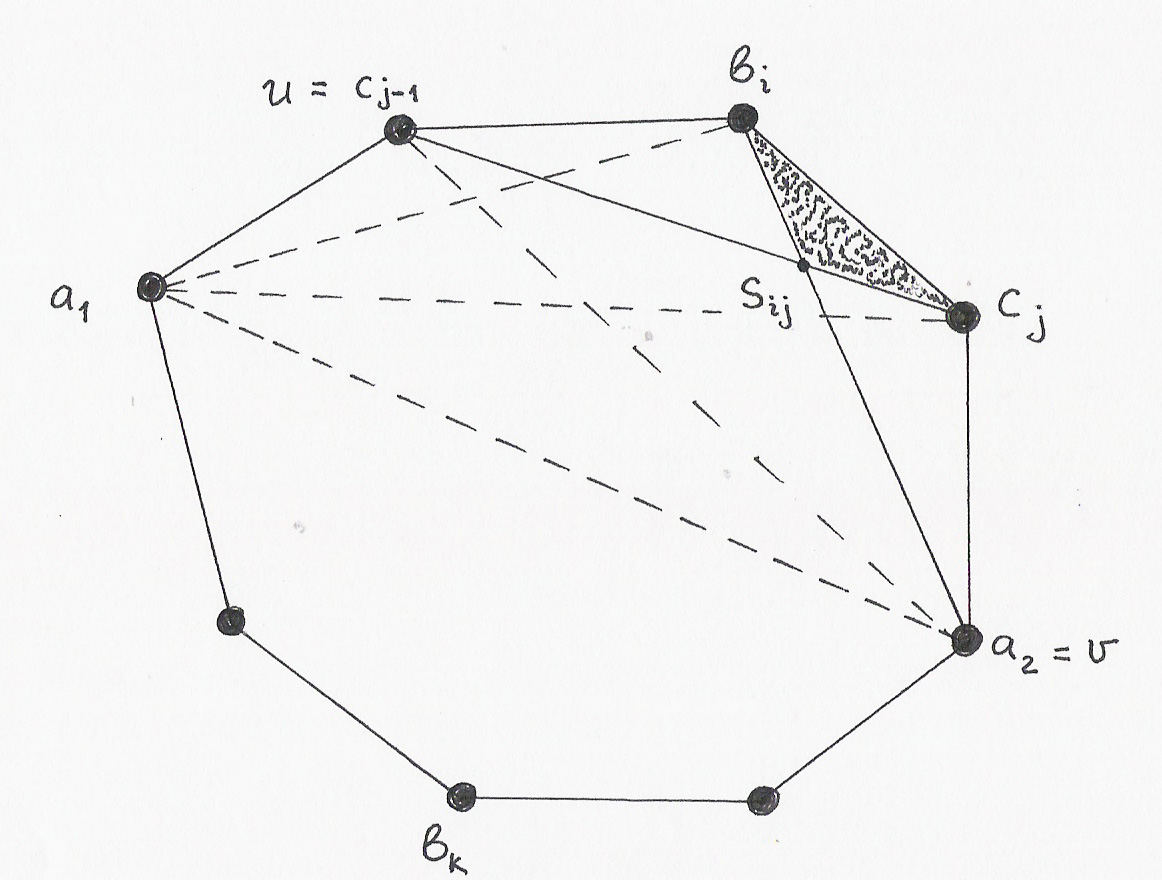}
\caption{}
\label{ABC}
\end{center}
\end{figure}

\begin{proof}
Assume without loss of generality that vertices $a_1,a_2 \in A$ are separated by points either from $B$ or $C$ in the circular order of vertices of polygon $M$.
If points from $B$ all belong to one semi-plane generated by line $(a_1,a_2)$, and all points from $C$ are in the other semi-plane, then every point from $M$ is in
$\overline{A\cup B}\cup\overline{A\cup C}$. Thus, assume that there are points from both $B$ and $C$ in one of semi-planes, and points from, say, $B$ are located in both semi-planes. Then the only points of $M$ that do not belong to
$\overline{A\cup B}\cup\overline{A\cup C}$ are the points of triangles of the form $\overline{\{b_i,c_j,s_{ij}\}}$, where $b_i \in B$, $c_j\in C$, $c_j$ immediately follows $b_i$ in the circular order of vertices of $M$, and $s_{ij}$ is the point of intersection of lines $(u,c_j)$, and $(b_i,v)$, where $u$ is  closest from $A\cup C$ preceding point
to $b_i$ in the circular order, and $v$ is the closest in circular order point from $A\cup B$ following $c_j$. 

According to the assumption, there is vertex $b_k \in B$ that belongs to the other semi-plane generated by $(a_1,a_2)$. Due to Lemma \ref{s in 3}, when $a_i$ is replaced by $b_i$, $a_1$ by $u$, $a_j$ by $c_j$, $a_k$ by $v$, $s$ by $s_{ij}$ and $a_s$ by $b_k$, it follows that $s_{ij}$ belongs to triangle $\overline{\{b_k,b_i,c_j\}}$. In particular, 
$\overline{\{b_i,c_j,s_{ij}\}}\subseteq \overline{\{b_k,b_i,c_j\}}\subseteq B\cup C$.  
\end{proof}

\emph{Proof of Theorem \ref{sharp}}.
Due to Lemma \ref{Sub}, $\G$ satisfies $2$-Carousel Rule, therefore, $y \in \overline{\{x,b,c\}}$. According to assumption that $\G$ is a subgeometry of $\G_0$, one can find an embedding $\phi$ of lattice of closed sets of $\G$ into lattice of closed sets of $\G_0$. Denote $U=\phi(\overline{u})$, for any $u \in \{\overline{a},\overline{b},\overline{c},\overline{x},\overline{y}\}$, and let $P=\phi(\emptyset)$.
Then, according to conditions of theorem,
$X,Y \subseteq \overline{A\cup B\cup C}$, $X \subseteq \overline{A\cup B\cup Y}$, $X\subseteq \overline{A \cup C\cup Y}$, $Y \subseteq  \overline{B \cup C \cup X}$. Besides, $P= Y \cap \overline{A\cup B} =Y \cap \overline{A\cup C}=Y\cap \overline{B\cup C}=Y\cap X$. 

Since points of $Y\setminus P$ are inside of convex polygon $\overline{A\cup B \cup C}$, but not in any $\overline{A\cup B},\overline{A\cup C},\overline{B\cup C}$, the vertices of $\overline{A\cup B \cup C}$ should appear in clusters, due to Lemma \ref{abc}: elements from $A$ should follow elements from $C$, which should follow elements
of $B$, in their circular order. Figure \ref{4} makes a sketch of arrangement, where $a_1$ and $a_2$ are end points of $A$-cluster, similarily, $b_1,b_2$ and $c_1,c_2$ are end-points of clusters $B$ and $C$, correspondingly. Elements of $Y\setminus P$  are located inside triangle formed by points of intersection of lines $(a_1,b_2)$, $(b_1,c_2)$ and $(c_1,a_2)$. We need to show that some point $x \in X\setminus P$ is in $\overline{Y\cup A}$.

\begin{figure}[ht]
\begin{center}
\includegraphics[scale=.45]{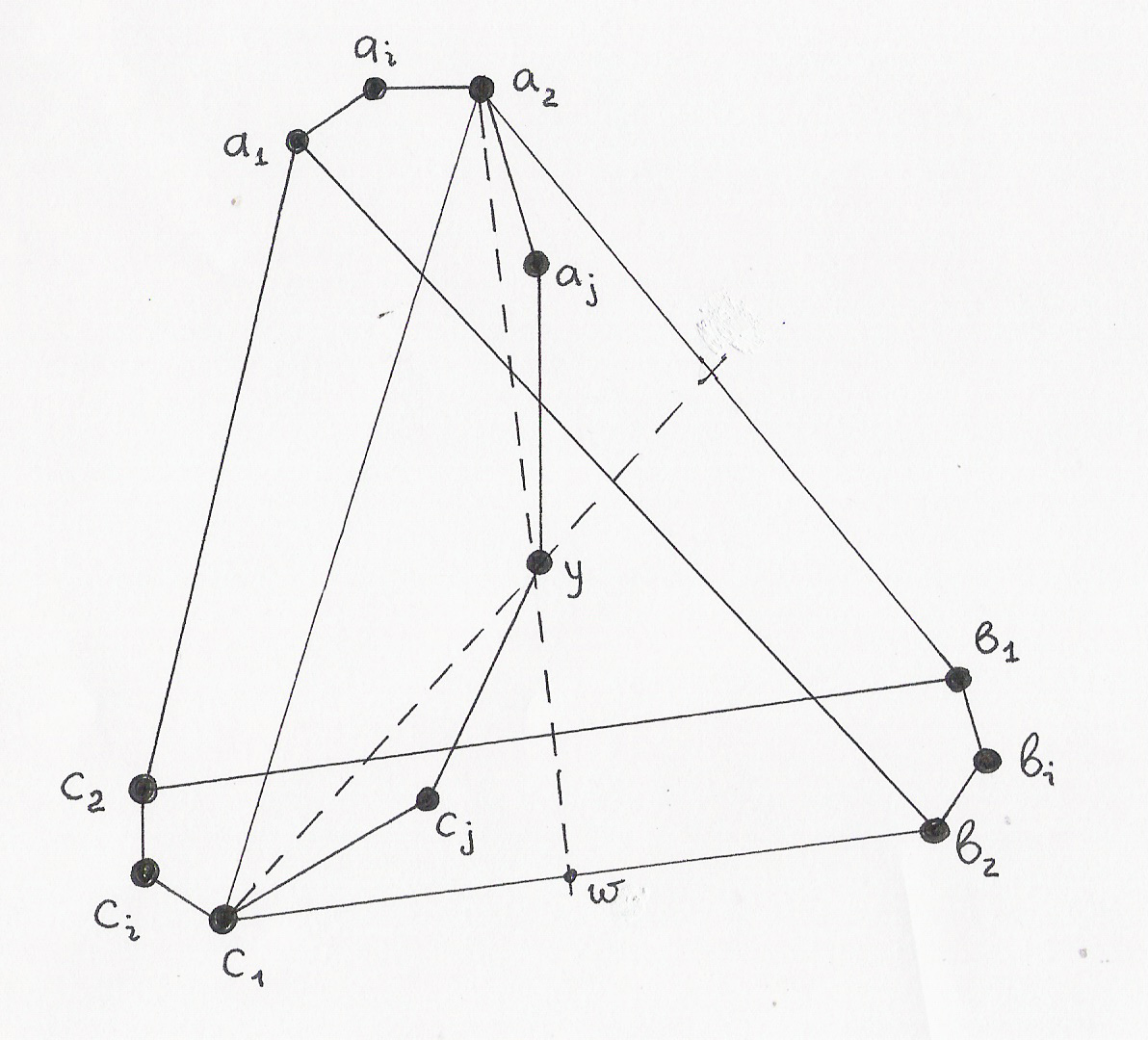}
\caption{}
\label{4}
\end{center}
\end{figure}

\begin{claim} Let points of $\overline{A\cup C\cup Y}$ follow each other in the circular order\\
$c_1,\dots,c_i,\dots, c_2,a_1,\dots, a_i,\dots,a_2,\dots,u,\dots,v\dots,c_1$. Then $v \in \overline{u\cup B\cup C}$.
\end{claim}
One can use Figure \ref{4} with possible identification of $u$ as $y$, and $v$ as $c_j$.
We assume that $u \not \in W= \overline{\{b_1,\dots,b_i,\dots,b_2,c_1,\dots,c_i,\dots,c_2\}}$, since otherwise the claim is obvious.

Draw the line $(u,a_2)$, then $v$ should be in the same semi-plane as $c_1$. Draw the line $(c_1,u)$, then $v$ should be in the semi-plane opposite to $a_2$. Since $u$ is an inner point of triangle $\overline{\{a_2,b_1,c_2\}}$, line $(a_2,u)$ crosses segment $[c_2,b_1]$ at some inner point. Hence, this line crosses another segment of convex polytope $W$, say, at point $w$ (this point does not necessarily belong to configuration that generates $G_0$). Thus,
$v$ belongs to the convex polytope formed by $u,c_1,w$ and all the vertices of the border of $W$ between $c_1$ and $w$. In particular, $v \in \overline{u \cup B\cup C}$, as desired. End of proof of Claim 1.

Since $\overline{A\cup C} \subset \overline{Y \cup A \cup C}$, some vertices of $\overline{Y \cup A \cup C}$ should be from $Y\setminus  \overline{A\cup C}= Y\setminus P=y\setminus \overline{B\cup C}$. Let $y_1$ be the first element from $Y\setminus P$ that appears  after $a_2$ in the circular order of vertices 
of $\overline{Y \cup A \cup C}$ given in Claim 1. According to Claim 1, no point from $C$ can appear between $a_2$ and $y_1$, since, otherwise, $y_1$ will be in
$\overline{B\cup C}$.  Thus, we have in the sequence from $a_2$ to $y_1$ only elements from $A$.

Similarly, let $y_2$ be the first element from $Y\setminus P=Y\setminus \overline{A\cup B}$ that appears in the circular order of vertices $b_2,\dots,b_1,a_2,\dots,a_1,\dots, b_2$ of
$\overline{Y\cup A\cup B}$ between $a_1$ and $b_2$. Then there is only elements from $A$ in this sequence between $a_1$ and $y_2$.

\begin{claim}
The sequences of segements forming the border of $\overline{Y\cup A\cup C}$ from $c_1$ to $a_2$ containing point $y_1$, and the border of $\overline{Y\cup A\cup B}$ from $b_2$ to $a_1$ containing point from $y_2$ intersect at some point
(not necessarily the point of configuration forming $G_0$).
\end{claim}
If only one vertex in $\overline{A\cup B\cup C}$ is from $A$, i.e. $a_1=a_2$, then $a_1$ might be the only point of intersection of $\overline{Y \cup A\cup B}$
and $\overline{Y \cup A\cup C}$. In this case we assume that this is point of intersection of borders of these two convex polygons stated in the claim.

Otherwise, points $c_1$ and $a_2$ are separated by line $(a_1,b_2)$.
If $c_1,t,\dots,s,w,u,v,$\\$\dots,a_2$ are the vertices of $\overline{Y\cup A\cup C}$, on the path from $c_1$ to $a_2$ that has point $y_1$, then one of the segments of this border, say, $[u,v]$, will cross $[a_1,b_2]$.  

Now $a_1$ and $b_2$ are separated by line $(u,v)$, hence, by all the lines
$(w,u)$,$(s,w) \dots,$ $(c_1,t)$. Therefore, the vertices of $\overline{Y\cup A\cup B}$ on the path from $b_2$ to $a_1$ and containing $y_2$, will cross each of those lines.
It should cross line $(c_1,t)$ on the ''right'' ray, i.e. on the ray with endpoint $c_1$ that contains $t$. On the other hand, it can cross line $(u,v)$ only on the ''left'' ray, i.e. on the ray with the endpoint $v$ that contains $u$. There should be a sequence of vertices in $\overline{Y\cup A\cup C}$, say, $s,w,u,v$, where the sequence of segements of $\overline{Y\cup A\cup B}$ will cross $(s,w)$ on the ''right'' ray, while it will cross $(u,v)$ on the ''left'' ray. This implies it will cross one of segments
$[s,w],[w,u],[u,v]$. End of proof of claim. 

Let us call a point of intersection from Claim 2 by $O$. Note again that, unlike points from $A,B,C$ or $Y$, point $O$ is just a geometrical location of intersection of some segments formed by points from $A\cup Y$. There are three possibilities for positioning of points $y_1,y_2$ and $O$ (see Figure \ref{5}). In all three cases,
$V=\overline{Y\cup A \cup B} \cap \overline{Y\cup A \cup C}$ is a convex polytope formed by points from $A$, point $O$ and  all the points prior to $O$ on the path from $a_1$ to $b_2$ and on the path from $a_2$ to $c_1$. According to the assumption, $X\subseteq V$ and $Y \subseteq \overline{X\cup B\cup C}$. We need to show that some point from $X\setminus P$ belongs to $\overline{Y\cup A}$.

\begin{figure}[ht]
\begin{center}
\includegraphics[scale=.45]{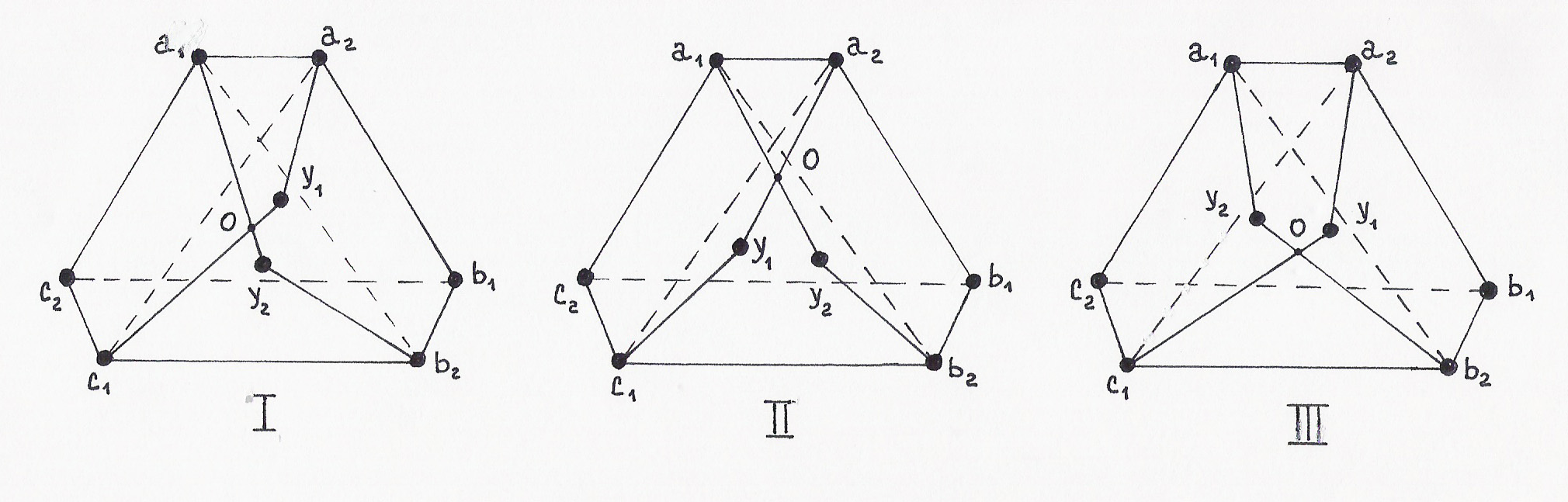}
\caption{}
\label{5}
\end{center}
\end{figure}

(I) On the path from $a_1$ to $b_2$, point $O$ occurs prior to $y_2$, but on the path from $a_2$ to $c_1$ point $O$ occurs after $y_1$. Evidently, $O$ belongs to 
$\overline{Y\cup A}$, since $O$ is on a segment connecting two points from $A\cup y_2$. We want to show that any vertex of $V$ between $O$ and $y_1$ (which is also a vertex of $\overline{Y\cup A\cup C}$)  cannot be from $C$. Indeed, if one vertex would be $c \in C$, then we can apply Claim 1 to vertex $c$ in place of $u$, and any vertex  of $\overline{Y\cup A \cup C}$ on the path from $c$ to $c_1$ in place of $v$. Then $v \in \overline{B\cup C}$. In particular, $O$ is in $\overline{B\cup C}$.
We can apply now a symmetric statement of Claim 1 to the points on the border of $\overline{Y\cup A \cup B}$, identifying $u$ with $O$ and $v$ with $y_2$. Then $y_2 \in \overline{B\cup C}$, a contradiction. 

Thus, all the vertices of $V$ must be in $\overline{Y\cup A}$, which proves $X\subseteq \overline{Y\cup A}$.

(II)  Both $y_1,y_2$ occur after $O$ on the corresponding paths. Then all the vertices of $V$ are in $\overline{Y\cup A}$, which is needed.

(III) Both $y_1,y_2$ occur prior to $O$ on the corresponding pathes. According to Claim 1, points of the path from $a_2$ to $c_1$ that appear after $y_1$ belong to  $\overline{y_1 \cup B \cup C}$, in particular, $O \in \overline{y_1\cup y_2\cup B \cup C}$, thus, the part of polytope $V$ formed by $O,y_1,y_2$ and all the vertices of both paths between
$y_1$ and $O$, and $y_2$ and $O$, correspondingly, belong to $\overline{y_1\cup y_2\cup B\cup C}$. If all the points from $X\setminus P$ would be in that part of $V$, we would have
$X \subseteq \overline{y_1 \cup y_2 \cup B\cup C}$. At least one of $y_1,y_2$ should be a vertex of $\overline{y_1 \cup y_2 \cup B\cup C}$. Then it can be in 
$\overline{X \cup B \cup C}$ only when it belongs to $X \cup B \cup C$. But then this points would be in $P$ due to $Y \cap X =P= Y \cap \overline{B \cup C}$, a contradiction. It follows that at least one point from $X\setminus P$ should be in the part of $V$ formed by points from $A \cup y_1 \cup y_2$.
Thus, $X \cap \overline{Y \cup A} > P$.  
\emph{End of proof of Theorem \ref{sharp}}\\

It follows from the proof of the theorem that the following property always holds in any geometry $G_0=\second$, hence, in any of its subgeometry:\\
\[
\text{For all closed sets } X,Y,A,B,C,
\]
\[
\text { if } Y \subseteq \overline{A\cup B\cup C}
\]
\[
Y\cap \overline{A\cup B}= Y\cap\overline{B\cup C}= Y\cap\overline{A\cup C}=Y\cap X=P <Y,X
\]
\[
X\subseteq \overline{Y\cup A\cup B}, X \subseteq \overline{Y\cup A\cup C} \text{ and } Y\subseteq \overline{X\cup B\cup C}
\]
\[
\text{then } X \cap \overline{\{A \cup Y\}}>P.
\] 

We will refer to this property as \emph{the Sharp 2-Carousel Rule}.\\
 
In conclusion of this section, we give an example of the convex geometry that satisfies $2$-Carousel Rule, but does not satisfy the Sharp $2$-Carousel Rule.\\

Let $A=\{a,b,c,x,y\}$ and the collection of closed sets of $(A,-)$ include all one-element and two-element subsets; besides, three-element subsets are $\{x,a,b\}$, 
$\{x,a,c\}$, $\{y,b,c\}$, $\{x,y,w\}$, for $w \in \{a,b,c\}$, and four-element are $\{a,b,x,y\}$, $\{b,c,x,y\}$, $\{a,c,x,y\}$. This implies $x,y \in \overline{\{a,b,c\}}$, $x \in \overline{\{y,a,b\}},\overline{\{y,a,c\}}$, and $y \in \overline{\{x,b,c\}}$. The Sharp 2-\textit{}Carousel Rule fails since $x \not \in \overline{\{y,a\}}$. Hence, $(A,-)$ is not a sub-geometry of any geometry of relatively convex sets.

\section{Concluding remarks}

Problem \ref{first} asks whether any of classes $\mathcal{C}_B$, $\mathcal{C}_f$ is universal for \emph{all} finite convex geometries. In fact, it is enough to check whether every finite \emph{atomistic} convex geometry is a sub-geometry in one of those classes.
Recall that a closure system $\A=(A,-)$ is called \emph{atomistic}, if all one-element subsets of $A$ are closed.
This follows from the result proved in \cite{AGT} (a different proof was given in \cite{AdNa}):
\begin{prop} Every finite convex geometry has a strong atomistic extension. In particular, every finite convex geometry is a sub-geometry of some atomistic convex geometry.
\end{prop}

On the other hand, for the description of sub-geometries of class $\mathcal{C}_n$, the proposition above is not of great help, due to the fact the strong atomistic extension  might not preserve the $n$-Carousel Rule.

Indeed, it is enough to give an example of an atomistic extension that does not preserve the $n$-Carath$\acute{\text{e}}$odory Property. 

Consider finite geometry $\G=(\{a,b,c,d,x\},-)$ given by its collection of closed sets $\mathcal{G}=\{\emptyset,a,b,d,ab,ad,bd,cd,abd,acd,abx,adx,bcd,bcdx,acdx,abdx,abcdx\}$. In this convex geometry, for any closed sets $\overline{U}=U, \overline{V}=V$,\\

$\overline{U\cup V} = \left\{
\begin{array}{c l}
  U\cup V \cup \{x\}, & a,b,c\in U\cup V; \\
  \\
  U\cup V, & \mbox{ otherwise}. \\
\end{array}
\right.
$

In particular, this convex geometry satisfies the $3$-Carath$\acute{\text{e}}$odory Property.

Let $\H=(\{a,b,c,d,x\},\tau)$ be another convex geometry on $\{a,b,c,d,x\}$, whose closed sets are all subsets of $\{a,b,c,d,x\}$, except $abcd$. One easily verifies that $\G$ is a sub-geometry of $\H$, therefore, $\H$ is an atomistic extension of $\G$. On the other hand, $3$-Carath$\acute{\text{e}}$odory property fails in $\H$, since $x \in \tau(abcd)$, but $x \not \in \tau(abc) \cup \tau(abd) \cup \tau(acd) \cup \tau(bcd)$. 

It would be interesting to describe necessary and sufficient properties of finite geometries which are sub-geometries of $n$-dimensional geometries of relatively convex sets.
One of main results in \cite{AW2} states that if a finite atomistic convex geometry with $k$ extreme points $a_1,\dots,a_k$ and points $x,y$ in the closure of $a_1,\dots,a_k$, satisfies the so-called Carousel Rule and Splitting Rule then it can be represented as $\alg{Co}(\mathbb{R}^2,A)$, with $A=\{a_1,\dots,a_k,x,y\}$ being some set of points on a plane. In this result the Carousel Rule is slightly more elaborate property than the $2$-Carousel Rule (a version of the Carousel Rule was also formulated in \cite{EdLa}, where the case of one point $x$ in the closure of $a_1,\dots,a_k$ was investigated).

At the moment we are not aware of any example of a finite convex geometry satisfying the $2$-Carousel Rule and the Sharp $2$-Carousel Rule but not representable by relatively convex sets on the plane.
Thus, we would like to ask:

\begin{pbm} Is every finite convex geometry that satisfies $2$-Carousel Rule and the Sharp $2$-Carousel Rule a sub-geometry of some (finite) geometry $\alg{Co}(\mathbb{R}^2,A)$?
\end{pbm}

In \cite{AW1}, the $2$-Carousel Rule was essential in establishing the correspondence between two problems: the representation of an atomistic convex geometry as $\alg{Co}(\mathbb{R}^2,A)$ and the realization of an order type by point configuration on the plane.
See \cite{GoPo} for the definition of an order type and \cite{Ha} for the recent overview and references on the topic.

{\bf Acknowledgements.} The paper was inspired by Maurice Pouzet and a prospect to celebrate his achievements at ROGICS'08 conference in Mahdia, Tunisia.
We are greateful to Nejib Zaguia and other organizers of the conference to make this event possible.

\end{document}